\documentclass[11pt,oneside]{amsart}
\usepackage{latexsym}
\usepackage{amsmath}
\usepackage{amssymb}

\usepackage{amsfonts}

\usepackage{amsfonts, amsmath, amssymb, amscd, amsthm, graphicx,enumerate}
\usepackage{epsfig, color}
\usepackage{hyperref}

\usepackage{epsf}
\usepackage{psfrag}

\newtheorem{theorem}{Theorem}

\newtheorem{thm}{Theorem}[section]

\newtheorem{cor}[thm]{Corollary}

\newtheorem{prop}[thm]{Proposition}
\theoremstyle{definition}
\newtheorem{defn}[thm]{Definition}

\newtheorem{conv}[thm]{Convention}

\newtheorem{prop-defn}[thm]{Proposition-Definition}

\begin{document}

\title[Ascending chain condition in generic groups]
      {Ascending chain condition in generic groups}
\author[I.~Kapovich]{Ilya Kapovich}

\address{Department of Mathematics and Statistics, Hunter College of CUNY\newline
  \indent 695 Park Ave, New York, NY 10065}
 
  \email{\tt ik535@hunter.cuny.edu}

\thanks{The author was supported by the individual NSF
  grants DMS-1905641}

\makeatletter
\@namedef{subjclassname@2020}{
	\textup{2020} Mathematics Subject Classification}
\makeatother


\date{\today}

\begin{abstract}
We prove that for any fixed integers $m\ge 2, t\ge 1, k\ge 2$ a generic $m$-generator $t$-relator group satisfies the Ascending Chain Condition for $k$-generated subgroups.
\end{abstract}
\maketitle

\section{Introduction}

For an integer $k\ge 1$ we say that a group $G$ is \emph{$k$-generated} if there exist elements $g_1,\dots, g_k\in G$ such that $G=\langle g_1,\dots, g_k\rangle$. Thus $G$ being $k$-generated is equivalent to having $rank(G)\le k$, where for a group $G$ the \emph{rank} of $G$, denoted $rank(G)$, is defined as $rank(G)=\min\{|S| : S\subseteq G, \langle S\rangle=G\}$. 

For $k\ge 1$ we say that a group $G$ satisfies the \emph{ascending chain condition $ACC_k$ for $k$-generated subgroups} if every strictly ascending chain of $k$-generated subgroups of $G$ 
\[
H_1\lneq H_2\lneq \dots
\]
terminates in finitely many steps. For example, if every nontrivial element of $G$ is contained in a unique maximal cyclic subgroup, then $G$ satisfies $ACC_1$. For this reason, every torsion-free word-hyperbolic group satisfies $ACC_1$. A finite group $G$ obviously satisfies $ACC_k$ for every $k\ge 1$. It is also fairly easy to see that a finitely generated abelian group satisfies $ACC_k$ for all $k\ge 1$. However, in most situations establishing condition $ACC_k$ requires substantial work.

A classic result of Higman~\cite{H51} and Takahasi shows that a free group $F$ satisfies $ACC_k$ for every $k\ge 1$.  Kapovich and Myasnikov~\cite{KM02} later gave a proof of this result using Stallings foldings~\cite{Sta}. Shusterman~\cite{Shu} used pro-finite techniques to show that limit groups (including, in particular, closed orientable surface groups) satisfy $ACC_k$ for all $k\ge 1$. Note that since closed non-orientable surface groups are commensurable with closed orientable surface groups, it follows that $ACC_k$ holds for all closed surface groups.
Recently Bering and Lazarovich~\cite{BL21} adapted the Kapovich-Myasnikov argument to give another proof of the $ACC_k$ condition for surface groups; they also proved that closed 3-manifold groups satisfy the ascending chain condition for $k$-generated free groups, with an arbitrary fixed $k\ge 1$.

 In this paper we establish condition $ACC_k$, for any fixed $k\ge 2$, for generic finite group presentations. We refer the reader to \cite{Oll05,Oll06,AO96,A1,A2,KS05,KS08} for a more detailed discussion of various models of random and generic groups. We only briefly recall a few relevant definitions here. Let $m\ge 2$, $A=\{a_1,\dots,a_k\}$ and $F_m=F(A)$. For an integer $t\ge 1$ let $\mathcal C_{m,t}$ be the set of all group presentations $\langle a_1,\dots, a_m|r_1,\dots, r_t\rangle$ where $r_i\in F(A)$ are nontrivial cyclically reduced words. For a subset $\mathcal P\subseteq \mathcal C_{m,t}$ and $n\ge 1$ denote by $\mathcal N(\mathcal P,n,t)$ the number of presentations $\langle a_1,\dots, a_m|r_1,\dots, r_t\rangle$ in $\mathcal P$ with $|r_i|\le n$ for $i=1,\dots, t$.  Similarly, we denote by $\mathcal S(\mathcal P,n,t)$ the number of presentations $\langle a_1,\dots, a_m|r_1,\dots, r_t\rangle$ in $\mathcal P$ with $|r_i|=n$ for $i=1,\dots, t$. 
 
 We say that $\mathcal P$ is \emph{generic} in $\mathcal C_{m,t}$ if
 \[
 \lim_{n\to\infty}\frac{\mathcal N(\mathcal P,n)}{\mathcal N(\mathcal C_{m,t},n,t)}=1.
 \]
If, in addition, the convergence in the above limit is exponentially fast, we say that $\mathcal P$ is \emph{exponentially generic} in $\mathcal C_{m,t}$. Note that both $\mathcal N(\mathcal C_{m,1},n,1)$ and $\mathcal S(\mathcal C_{m,1},n,1)$  grow as $Const\, (2m-1)^n$, and both $\mathcal N(\mathcal C_{m,t},n,t)$ and $\mathcal S(\mathcal C_{m,t},n,t)$ grow as $Const\, (2m-1)^{tn}$. For this reason replacing $|r_i|\le n$ by $|r_i|=n$ leads to the notion of genericity with similar properties.

Our main result is:

\begin{theorem}\label{t:A}
Let $m\ge 2, t\ge 1, k\ge 1$ be integers. There exists an exponentially generic class $Q_{m,t,k}$ of group presentations with $m$ generators and $t$ defining relators such that every group $G=\langle a_1,\dots, a_m| r_1,\dots, r_t  \rangle$ from $Q_{m,t,k}$ satisfies the Ascending Chain Condition $ACC_k$ for $k$-generated subgroups.
\end{theorem}

Theorem~\ref{t:A} deals with the "basic" or "few relators" model of random groups, where the number $t\ge 1$ of defining relators is fixed and $n=\max_i |r_i|$ tends to infinity. The "density" model of random groups, for a fixed \emph{density} parameter $d\in ),1)$, considers group presentations $G=\langle a_1,\dots, a_m| r_1,\dots, r_t  \rangle$ where all $|r_i|=n$ and where $t=t_n$ grows with $n$ as $t_n=(2m-1)^{dn}$.  We denote $\mathcal C_m$ the set of all finite group presentations $\langle a_1,\dots, a_m|R\rangle$ where $R$ is a finite collection of nontrivial cyclically reduced words of equal length. Let $\mathcal P\subseteq \mathcal C_m$. We say that a presentation from $\mathcal C_m$  \emph{belongs to $\mathcal P$ with overwhelming probability} at density $0<d<1$ if

\[
 \lim_{n\to\infty}\frac{\mathcal S(\mathcal P,n,t_n)}{\mathcal S(\mathcal C_{m},n, t_n)}=1,
 \]
where $t_n=\lfloor (2m-1)^{dn}\rfloor$.

\begin{theorem}\label{t:B}
Let $m\ge 2, k\ge 1$ be integers. 

There exists $0<d_0=d_0(m,k)<1$ such that for every $0<d\le d_0$ with overwhelming probability a group $G$ given by a density-$d$ presentation on $m$ generators satisfies the Ascending Chain Condition $ACC_k$ for $k$-generated subgroups.
\end{theorem}

All of the $AC_k$ groups produced by the proofs of Theorem~\ref{t:A} and Theorem~\ref{t:B} satisfy the $C'(1/6)$ small cancellation condition and therefore are word-hyperbolic. Note that there do exist torsion-free word-hyperbolic groups where already $AC_2$ fails. Consider for example, the group 
\[
G=\langle a, b, t| t^{-1}at=ab^2a, t^{-1}bt=ba^2b\rangle.
\]
This group arises as the mapping torus of the injective non-surjective endomorphism $\phi:F(a,b)\to F(a,b)$, $\phi(a)=ab^2a$, $\phi(b)=ba^2b$. The subgroup $\phi(F(a,b))\lneq F(a,b)$ is malnormal in $F(a,b)$ and $\phi$ is an expanding immersion. Then it follows from the Bestvina-Feighn Combination Theorem~\cite{BF92} that $G$ is word-hyperbolic (see \cite{K00} for details). Put $H_n =t^n F(a,b) t^{-n}$ where $n=0,1, 2,\dots$.
Then
\[
H_0\lneq H_1 \lneq H_2\lneq \dots
\]
is an infinite strictly ascending chain of subgroups of of $G$ where each $H_i$ is free of rank $2$. Thus $ACC_2$ fails for $G$. 

%Note that in this example the subgroup $H_0=F(a,b)\le G$ is not quasiconvex in $G$. By contrast, for the groups $G$ produced by the proofs of Theorem~\ref{t:A} and Theorem~\ref{t:B} all $k$-generated subgroups are quasiconvex. It seems reasonable to conjecture that if $G$ is a torsion-free word-hyperbolic group where all $k$-generated subgroups are quasiconvex then $G$ satisfies $ACC_k$.

\section{Representing subgroups by labeled graphs}

\begin{conv}
 For the remainder of this paper, unless specified otherwise, let $A=\{a_1,\dots, a_m\}$ be a finite alphabet, where $m\ge 2$, which will be the
  marked set of generators of the group $G$ under consideration.  We denote $F_m=F(A)=F(a_1,\dots, a_m)$, the free group on $A$.
\end{conv}

Following the approach of Stallings~\cite{Sta}, we use labeled graphs
to study finitely generated subgroups of quotients of $F_m$.
\cite{Sta,Sch,AO,KM}. 

By a \emph{graph} we mean a 1-dimensional CW-complex $\Gamma$. We refer to $0$-cells of $\Gamma$ as \emph{vertices} and to open $1$-cells of $\Gamma$ as \emph{topological edges}. An \emph{oriented edge} of $\Gamma$ is a topological edge with a choice of an orientation on it. For an oriented edge $e$ we denote by $e^{-1}$ the same edge with the opposite orientation. We denote the set of all vertices of $\Gamma$ by $V\Gamma$ and the set of all oriented edges of $\Gamma$ by $E\Gamma$. For an oriented edge $e\in E\Gamma$ the attaching maps define its \emph{initial vertex} $o(e)\in V\Gamma$ and its terminal vertex $t(e)\in E\Gamma$.  For a vertex $v\in \Gamma$ the \emph{degree} $deg_\Gamma(v)$ of $v$ is the number of all $e\in E\Gamma$ with $o(e)=v$. An \emph{edge-path} in $\Gamma$ is a sequence of edges $\gamma=e_1,\dots, e_k$ (where $k\ge 0$) of oriented edges of $\Gamma$ such that $t(e_i)=o(e_{i+1})$ for $1\le i<k$. We put $o(\gamma)=o(e_1)$, $t(\gamma)=t(e_k)$, $|\gamma|=k$ and $\gamma^{-1}=e_k^{-1},\dots, e_1$. For $k=0$ we view $\gamma=v\in V\Gamma$ as an edge-path with $o(\gamma)=t(\gamma)=v$, $|\gamma|=0$ and $\gamma^{-1}=\gamma=v$. We say that an edge-path $\gamma$ is \emph{reduced} if it has no subpaths of the form $e,e^{-1}$, where $e\in E\Gamma$. An \emph{arc} in a graph $\Gamma$ is a simple edge-path of positive length, possibly closed, where every intermediate vertex of the path has degree 2 in
$\Gamma$.  Thus in a finite graph $\Gamma$, every arc is contained in a unique \emph{maximal arc}, whose end-vertices have degree $\ge 3$ or degree $1$.

\begin{defn}
  An \emph{$A$-graph} $\Gamma$ consists of an underlying oriented
  graph where every (oriented) edge $e$ is labeled by an element
  $\theta(e)\in A^{\pm 1}$ in such a way that $\theta(e^{-1})=\theta(e)^{-1}$ for every edge
  $e$ of $\Gamma$. We allow multiple edges between vertices as well as
  edges which are loops.

  An $A$-graph $\Gamma$ is said to be \emph{non-folded} if there
  exists a vertex $x$ and two distinct edges $e_1, e_2$ with origin
  $x$ such that the words $\theta(e_1)=\theta(e_2)$.  Otherwise
  $\Gamma$ is said to be \emph{folded}.
\end{defn}

If $\gamma=e_1,\dots, e_k$ is an edge-path in an $A$-graph $\Gamma$, we put $\theta(\gamma)=\theta(e_1)\dots \theta(e_k)$. Thus $\theta(\gamma)$ is a word in $A^{\pm 1}$. For a path $\gamma$ with $|\gamma|=0$ we put $\theta(\gamma)$ to be the empty word $\epsilon$.

For a finite graph $\Gamma$ we denote by $b(\Gamma)$ the first Betti number of $\Gamma$. Thus if $\Gamma$ is connected then $b(\Gamma)=rank(\pi_1(\Gamma))$. 
For an edge-path $p$ in $\Gamma$ we denote by $o(p)$ the initial vertex of $p$ and by $t(p)$ the terminal vertex of $p$ in $\Gamma$. 

Every edge-path $p$ in an $A$-graph $\Gamma$ has a label $\theta(p)$ which is a word. The number of edges in $p$
will be called the \emph{length} of $p$ and denoted $|p|$.  For $F_m=\overline{F}(A)$
a path $p$ in an $A$-graph $\Gamma$ is said to be \emph{reduced}
if it does not contains subpaths of the form $e, e^{-1}$ and if
$p$ does not contain subpaths of the form $e,e$ where $e$ is an
edge of $\Gamma$.

The definition implies that an $A$-graph $\Gamma$ is folded if and only if the label of every reduced edge-path in $\Gamma$ is a reduced word.

\begin{defn}
Let $G=\langle A|R\rangle=\langle a_1,\dots, a_m| r_1, r_2\dots \rangle$ be a group presentation on the generators $a_1,\dots, a_m$, where each $r_i$ is a cyclically reduced word in $F(A)$.

Let $\Gamma$ be a connected $A$-graph with a base-vertex $x_0$. There is a natural labeling homomorphism $\Theta: \pi_1(\Gamma,x_0)\to G$ where for every closed edge-path $\gamma$ from $x_0$ to $x_0$ we put $\Theta([\gamma])=\theta(\gamma)\in G$.

We say that the subgroup $H=\Theta(\pi_1(\Gamma,x_0))\le G$ is \emph{represented} by $(\Gamma,x_0)$. 
\end{defn}

Note that if $H$ is represented by $(\Gamma,x_0)$, where $\Gamma$ is a finite connected $A$-graph, then $rank(H)\le b(\Gamma)$.

%We sometimes also will need the following move on $A$-graphs:
%\begin{defn}[Removing a degree-one vertex]
%  Suppose $e$ is an edge of an $A$-graph $\Gamma$ such that the vertex
%  $t(e)$ has degree one in $\Gamma$. We remove the edge $e$ and the
%  vertex $t(e)$ from $\Gamma$.
%\end{defn}

In addition to foldings, we need the following transformation of
labeled graphs used by Ol'shanskii and Arzhantseva \cite{AO96}.

\begin{defn}[Fold]
Let $\Gamma$ be a connected $A$-graph and let $e_1,e_2$ be two distinct oriented edges in $\Gamma$ with $o(e_1)=o(e_2)$ and $\theta(e_1)=\theta(e_2)=a\in A^{\pm 1}$.

A \emph{Stalling fold} on $e_1,e_2$ consists in producing a new $A$-graph $\Gamma'$ obtained from $\Gamma$ by identifying $e_1,e_2$ into a single oriented edge $e$ with label $\theta(e)=a$ (and identifying $t(e_1)$ and $t(e_2)$ into a single vertex if $t(e_1)\ne t(e_2)$ in $\Gamma$.  This fold is called \emph{singular} if $t(e_1)=t(e_2)$ and \emph{non-singular} if $t(e_1)=t(e_2)$. 
\end{defn}

\begin{defn}[Arzhantseva-Ol'shanskii move]\label{defn:AO} Let $p=p_1 p' p_2$ be a reduced edge-path in a finite connected
  $A$-graph $\Gamma$ such that $p'$ is an arc of $\Gamma$ and the
  paths $p_1$, $p_2$ do not overlap $p'$. Let $p$ have initial vertex
  $x$, terminal vertex $y$, and label $\theta(p)=v$.  Let $z$ be a
  reduced word such that $v =_G z$ in $G$.
  
  We now modify $\Gamma$ by adding a new arc $q$ from $x$ to $y$ with
  label $z$ and removing all the edges and interior vertices of $p'$ from $\Gamma$.
  
  We will say that the resulting $A$-graph $\Gamma'$ is obtained from
  $\Gamma$ by an \emph{$AO$-move} on the arc $p'$.
  
Note that we allow the case where $z=1$ is the trivial word, that is, where $v=_G 1$. In this case the $AO$ move removes the interior of $p'$ from $\Gamma$ and, if $x\ne y$ in $\Gamma$, identifies $x$ and $y$ into a single vertex. If we already had $x=y$ in $\Gamma$ then in this case the $AO$ move just removes the interior of $p'$ from $\Gamma$; we refer to the $AO$ moves of this latter type (where $z=1$ and $x=y$ in $\Gamma$) as \emph{singular}, and otherwise call an $AO$-move \emph{non-singular}. 
  
\end{defn}

We record here some important properties of moves on $A$-graphs that follow directly from the definitions (see also~\cite{AO96,KS05}).

\begin{prop}\label{p:AO-move}
Let $\Gamma$ be a finite connected $A$-graph with a base-vertex $x_0$ such that $(\Gamma,x_0)$ represents a subgroup $H\le G$.

\begin{enumerate}
\item Let $\Gamma'$ be obtained from $\Gamma$ by a fold and let $x_0'$ be the image of $x_0$ in $\Gamma'$. Then $(\Gamma',x_0')$ also represents $H\le G$, and $b(\Gamma')\le b(\Gamma)$. Moreover, $b(\Gamma')=b(\Gamma)$ if the fold is non-singular, and $b(\Gamma')=b(\Gamma)-1$ if the fold is singular.
\item Let $\Gamma'$ be obtained from $\Gamma$ by an AO-move on an arc $p'$ such that $x_0$ is not an interior vertex of $p'$.  [If the $AO$-move was on an arc $p'$ with label $v=_G 1$ with endpoints $x\ne y$ and if $x_0\in \{x,y\}$, we still denote the image of $x_0$ in $\Gamma'$ by $x_0$.]

Then $(\Gamma',x_0)$ also represents $H\le G$, and $b(\Gamma')\le b(\Gamma)$. Moreover, $b(\Gamma')= b(\Gamma)$ if the AO-move is nonsingular, and $b(\Gamma')= b(\Gamma)-1$ if the AO-move is singular.

\end{enumerate}

\end{prop}

\section{Equality diagrams in small cancellation groups}

Let $G$ be a group with a finite generating set $A$. As usual, for a word $w$ over $A^{\pm 1}$ we denote by $\overline w$ the element of $G$ represented by $A$. We also denote by $d_A$ the word metric on $G$ corresponding to $A$.
Recall that for $C\ge 1$, a word $w$ over $A^{\pm 1}$ is called a $(C,0)$-quasigeodesic in $G$ if for every subword $v$ of $w$ we have $|v|\le C d_A(1, \overline v)$.  We refer the reader to \cite[Ch. V]{LS} for background info on small cancellation groups.

\begin{defn}[$\lambda$-reduced words]
Let 
\[ 
G=\langle a_1,\dots, a_m | R\rangle 
\] be a $C'(\lambda)$-presentation, where $0<\lambda\le 1/6$. 
  
  A word $w$ over $A^{\pm 1}$ is called \emph{$\lambda$-reduced} if $w$ is freely reduced and if whenever $w$ contains a subword $v$ such that $v$ is also a subword of some cyclic permutation of $r$ or $r^{-1}$ for some $r\in R$ then $|v|\le (1-3\lambda)|r|$.
\end{defn}
Note that of $\lambda=1/6$ then being $\lambda$-reduced is the same as being $\lambda$-reduced. In general, for $0<\lambda\le 1/6$ being Dehn-reduced implies being $\lambda$-reduced. 

Greendlinger's lemma (\cite[Theorem 4.4, Ch. V]{LS}) implies:

\begin{prop}\label{p:green}
Let \[ 
G=\langle a_1,\dots, a_m | R\rangle 
\] be a $C'(\lambda)$-presentation, where $0<\lambda\le 1/6$. 
Let $w$ be a nontrivial freely reduced word over $A^{\pm 1}$ such that $w=_G 1$. Then $w$ is not $\lambda$-reduced.
\end{prop}

Thus for a
$C'(\lambda)$-presentation with $\lambda\le 1/6$, a nontrivial $\lambda$-reduced word in $F(a_1,\dots, a_m)$ represents a
nontrivial element of $G$.

The following proposition follows from the basic results of small cancellation
theory, established in Ch. V, Sections 3-5 of \cite{LS}. The proof of this proposition is essentially identical to the proof of Lemma~2.11 in \cite{KS04}
(and is similar to the proof of Proposition~39 in \cite{Stre}). For these reasons we omit the details.

\begin{prop}\label{p:e}[Equality diagrams in
  $C'(\lambda)$-groups]

Let \[ G=\langle a_1,\dots, a_m |
  R\rangle\tag{$\ast$} \] be a $C'(\lambda)$-presentation, where
  $\lambda\le 1/6$.

Let $w_1,w_2\in F(a_1,\dots, a_m)$ be freely reduced and
  $\lambda$-reduced words such that $w_1=_G w_2$.

  Then any reduced van Kampen diagram $D$ over $(\ast)$, realizing the equality $w_1=_G
  w_2$, has the form as shown in Figure~\ref{Fi:equ}. Specifically, any
  region $Q$ of $D$ labeled by $r\in R$ intersects both the upper
  boundary of $D$ (labeled by $w_1$) and the lower boundary of $D$
  (labeled by $w_2$) in simple segments $\alpha_1, \alpha_2$
  accordingly, satisfying \[ \lambda|r|\le |\alpha_j|\le
  (1-3\lambda)|r|.  \] Moreover, if two regions $Q,Q'$ in $D$,
  labeled by $r,r'\in R$, have a common edge, then they intersect in
  closed simple segment $\gamma$ joining a point of the upper boundary
  of $D$ with a point of the lower boundary of $D$ and labeled by a
  piece with respect to $R$.  In particular $|\gamma|<\lambda |r|$ and
  $|\gamma|< \lambda|r'|$.

\begin{figure}[htb]
		\scalebox{.9}{\includegraphics{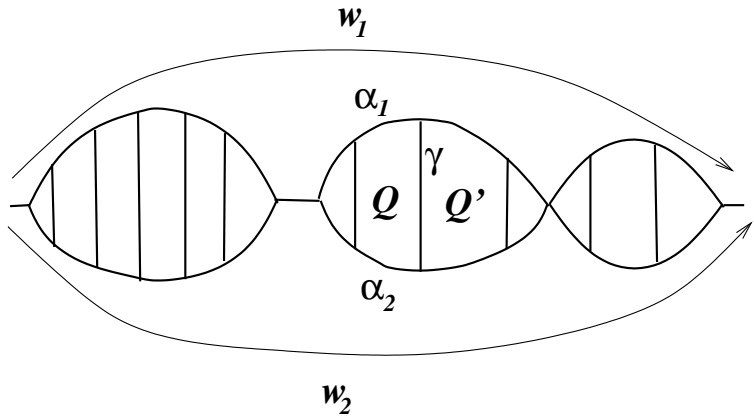}}
		\caption{Equality diagram in a small cancellation
                  group}
\label{Fi:equ}

\end{figure}

\end{prop}

Proposition~\ref{p:e} immediately implies:

\begin{cor}\label{c:q}
Let Let \[ G=\langle a_1,\dots, a_m |
  R\rangle \] be a $C'(\lambda)$-presentation, where
  $0<\lambda\le 1/6$. Put $C=\frac{1-3\lambda}{\lambda}$. [Note that $0<\lambda\le 1/6$ implies $C\ge 1$.]
  
  Then every $\lambda$-reduced word $w\in F(A)$ is a $(C,0)$-quasigeodesic in the Cayley graph of $G$ with respect to $A$.

\end{cor}

%\begin{figure}
%\centering
%\includegraphics{equal.pdf}
%\caption{Equality diagram in a small cancellation group}\label{Fi:eq}
%\end{figure}

\section{Genericity conditions}

We will need some genericity consitions introduced by Arzhansteva and Ol'shanskii~\cite{AO96} and further explored in \cite{A1,A2, KS05}.
Recall that $m\ge 2$, $A=\{a_1,\dots, a_m\}$ and $F_m=F(A)$ are fixed.

%\begin{defn}\cite{AO96}
%  Let $0< \mu< 1$ be a real number. A nontrivial freely reduced word $w$ in
%  $F(A)=F(a_1,\dots, a_m)$  is called
 % \emph{$\mu$-readable} if there exists a connected finite folded connected $A$-graph
 % $\Gamma$ such that:

%\begin{enumerate}
%\item The number of topological edges in $\Gamma$ is at most $\mu |w|$.
%\item We have $b(\Gamma)\le m-1$.
%\item There is a reduced path in $\Gamma$ with label $w$.
%\end{enumerate}
%\end{defn}

The following genericity condition was defined by Arzhantseva in \cite{A1} and generalizes an earlier version of this condition defined by Arzhantseva and Ol'shansky in \cite{AO96}.

\begin{defn}\cite{A1}
  Let $0< \mu< 1$ be a real number and let $k\ge 2$ be an integer. A
  nontrivial freely reduced word $w$ in $F(A)$  is called
  \emph{$(\mu,k)$-readable} if there exists a finite connected folded
  $A$-graph $\Gamma$ such that:

\begin{enumerate}
\item The number of edges in $\Gamma$ is at most $\mu |w|$.
\item We have $b(\Gamma)\le k$. 
\item There is a path in $\Gamma$ with label $w$.
\item The graph $\Gamma$ has at least one vertex of degree $< 2m$.
\end{enumerate}
\end{defn}
Note that if $k\le m-1$ then any finite folded $A$-graph $\Gamma$ with $b(\Gamma)\le k$ has a vertex of degree $<2m$, so that (2) is redundant in this case.

\begin{defn}\label{LML}
  Let $0< \mu<1, 0<\lambda\le 1/6$ be a real numbers, let $t\ge 1$ and $k\ge m$ be integers.
 We will say that a tuple of nontrivial cyclically reduced words $(r_1,\dots, r_t)$ in $F(A)$ satisfies the \emph{$(\lambda, \mu,
   k)$-condition} if:

\begin{enumerate}
\item The words $r_i$ are  not a proper powers in $F(A)$.
%\item For $i\ne j$ $r_i$ is not conjugate to $r_j^{\pm 1}$ in $F(A)$.
\item The symmetrized closure of $\{r_1,\dots, r_t\}$ satisfies the $C'(\lambda)$ small
  cancellation condition.
\item If $w$ is a subword of a cyclic permutation of some $r_i^{\pm 1}$ and
  $|w|\ge |r_i|/2$ then $w$ is not $(\mu,k)$-readable.
\end{enumerate}

For a fixed integer $t\ge 1$ we denote by $Q_{m,t}(\lambda,\mu,k)$ the set of all finite presentations $\langle a_1,\dots, a_m| r_1,\dots, r_t\rangle$ satisfying the $(\lambda, \mu, k)$-condition.
We also denote $Q_{m}(\lambda,\mu,k)=\cup_{t=1}^\infty Q_{m,t}(\lambda,\mu,k)$.

\end{defn}

The results of Arzhantseva-Ol'shanskii~\cite{AO96} and Arzhantseva~\cite{A1} imply:

\begin{prop}\label{p:generic}
Let $m\ge 2, t\ge 1, k\ge m$ be integers. Let $0<\lambda, \mu<1$ be such that
\[
\lambda<\frac{\mu}{15k+3\lambda}<\frac{1}{6}.
\]
Then $Q_{m,t}(\lambda,\mu,k)$ is exponentially generic in the set of all presentations $\langle a_1,\dots, a_m| r_1,\dots, r_t\rangle$ where $r_i$ are cyclically reduced words in $F(A)$.
\end{prop}

\section{Properties of subgroups of generic groups}

\begin{prop}\label{p:red}
Let $m\ge 2, t\ge 1, k\ge m$ be integers and let $0<\lambda, \mu<1$ be real numbers satisfying
\[
\lambda<\frac{\mu}{15k+3\lambda}<\frac{1}{6} \quad \text{ and } \quad
0<\frac{\lambda}{1-3\lambda}< \frac{\mu}{15k+5}.
\]
Let \[ G=\langle a_1,\dots, a_m| r_1,\dots, r_t\rangle \tag{\dag} \] be a presentation satisfying the $Q_{m,t}(\lambda,\mu,k)$ condition.
Let $H\le G$ be a $k$-generated subgroup with $[G:H]=\infty$. Let $(\Gamma,x_0)$ be an $A$-graph with $b(\Gamma)\le k$ representing $H\le G$ with the smallest number of edges among all such graphs.

Then $\Gamma$ is folded and the label of every reduced edge-path in $\Gamma$ is $\lambda$-reduced with respect to presentation $(\dag)$.
\end{prop}

\begin{proof}
We may assume that $H\ne 1$ since otherwise the result holds vacuously. By the minimal choice of $\Gamma$, the $A$-graph $\Gamma$ is folded and has no degree-1 vertices except possibly $x_0$. Since $H\le G$ has infinite index, there exists a vertex of degree $<2m$ in $\Gamma$.

Suppose that the conclusion of the proposition fails for $\Gamma$. Then there exists a reduced edge-path $p$ in $\Gamma$ with label $v$ such that $v$ is a subword of a cyclic permutation $r$ of some defining relator $r_j^{\pm 1}$ with $|v|>(1-3\lambda)|r|$. The maximal subarcs of $\Gamma$ break $p$ into a concatenation $p=p_1\dots p_s$, where $p_2,\dots, p_{s-1}$ are maximal arcs in $\Gamma$ and where $p_1,p_s$ are contained in maximal arcs.

Recall that by definition every endpoint of a maximal arc in $\Gamma$ either has degree $\ge 3$ or is the base-vertex $x_0$. Thus $x_0$ is not an interior vertex of  any of $p_1,\dots, p_s$. 

There are two cases to consider.

{\bf Case 1.} Suppose there exists some $p_i$ with $|p_i|>5\lambda|r|$. 

We claim that there exists a subpath $p_i'$ of $p_i$ with $|p'|>3\lambda|r|$ such that $p_i'$ does not overlap the rest of the path $p$.

Assume first that $1<i<s$. Since $|p_i|>5\lambda|r|$, the word $r$ is not a proper power and since $(\dag)$ satisfies the $C'(\lambda)$ small cancellation condition, we have $p_i\ne p_j^{\pm 1}$ for $1<j<s, j\ne i$ and hence $p_i$ does not overlap any such $p_j$. Similarly, the overlap of $p_i$ with $p_1, p_s$ have length $< \lambda|r|$ each. Therefore $p_i$ has a subpath of length $>3\lambda|r|$ that does not overlap with the rest of $p$, as claimed.
Assume now that $i=1$ (the case $i=s$ is similar).  Thus $|p_1|>5\lambda|r|$. Then $p_1$ does not overlap any $p_j$ with $1<j<s$ since otherwise we are in the previous situation. The overlap of $p_1$ and $p_s$ has length $<\lambda|r|$ since $(\dag)$ satisfies $C'(\lambda)$  and since $r$ is not a proper power.
Thus the claim is verified.

Let $\Gamma'$ be obtained from $\Gamma$ by performing the AO-move on the arc $p_i'$ corresponding to the relator $r$. Since $|p'|>3\lambda|r|$, the graph $\Gamma$ has fewer edges than $\Gamma$. Moreover, since $x_0$ was not an interior vertex of $p_i'$, we have $x_0\in V\Gamma'$ and by Proposition~\ref{p:AO-move} the graph $\Gamma'$ represents $H\le G$. Moreover, $b(\Gamma')\le b(\Gamma)\le k$. This contradicts the minimal choice of $\Gamma$.

Thus Case~1 is impossible.

 {\bf Case 2.} Suppose that $|p_i|\le 5\lambda|r|$ for $i=1,\dots, s$. 
 
 Let $\Gamma''$ be the subgraph of $\Gamma$ spanned by $p=p_1\dots p_s$. Since $b(\Gamma)\le k$ and $\Gamma$ has at most one vertex of degree $1$ (namely $x_0$), there are at most $3k+1$ maximal arcs in $\Gamma$. Hence the number of edges $\Gamma''$ is at most
 \[
 (3k+1)5\lambda|r|\le \mu (1-3\lambda)|r|\le \mu|v|,
\]  
where the inequality $(3k+1)5\lambda\le  1-3\lambda$ holds by our choice of $\lambda, \mu, k$. 
Since $\Gamma$ has a vertex of degree $<2m$ and all vertices in $\Gamma$ have degree $\le 2m$, it follows that $\Gamma''$ also has a vertex of degree $<2m$. Finally, since $\Gamma''$ is a connected subgraph of $\Gamma$ and $b(\Gamma)\le k$, it follows that $b(\Gamma'')\le k$. Hence the word $v$ is $(\mu,k)$-readable. Since $v$ is a subword of $r$ with $|v|\ge |r|/2$, we get a contradiction with the assumption that $(\dag)$ satisfies the $Q_{m,t}(\lambda,\mu,k)$-condition.

Therefore Case~2 is also impossible.

Hence the label of every reduced path in $\Gamma$ is $\lambda$-reduced, as required.

\end{proof}

Proposition~\ref{p:red} and Corollary~\ref{c:q} imply:

\begin{cor}\label{c:red}
Let $m,t,k,\lambda, \mu, G, H$ be as in Proposition~\ref{p:red}. Then the following hold:

\begin{enumerate}
\item The labeling homomorphism $\phi: \pi_1(\Gamma, x_0)\to G$ is injective. In particular, the group $H=\phi(\pi_1(\Gamma, x_0) )$ is free.
\item The label of every reduced edge-path in $\Gamma$ is a $(C,0)$-quasigeodesic in $G$ with $C=\frac{1-3\lambda}{\lambda}$. 
\item The subgroup $H\le G$ is quasiconvex in $G$. 
\end{enumerate}
\end{cor}

\section{Proofs of the main results}

\begin{thm}\label{t:ACC}
Let $k\ge m\ge 2, t\ge 1$ integers and let $0<\lambda,\mu<1$ be real numbers such that
\[
\lambda<\frac{\mu}{15k+3\lambda}<\frac{1}{6}
\quad
\text{ and }
\quad
0<\frac{\lambda}{1-3\lambda}< \frac{\mu}{15k+5}.
\]
Let \[ G=\langle a_1,\dots, a_m| r_1,\dots, r_t\rangle \tag{\dag} \] be a presentation satisfying the $Q_{m,t}(\lambda,\mu,k)$ condition.
Then $ACC_k$ holds for $G$.
\end{thm}

\begin{proof}

Let $G=\langle a_1,\dots, a_m| r_1,\dots, r_t\rangle$ be a group presentation satisfying the $(\lambda, \mu, k)$ condition. 
Put $C=\frac{1-3\lambda}{\lambda}$.

We claim that $G$ satisfies $ACC_k$.  Indeed, suppose not.

Then there exists an infinite strictly ascending sequence 

\[
H_1\lneq H_2\lneq H_3 \lneq \dots \tag{\ddag}
\]
of nontrivial $k$-generated subgroups of $G$.  Thus every $H_i$ has infinite index in $G$ since otherwise a sequence as above cannot be infinite.

For every $i=1, 2, 3,\dots$ let $(\Gamma_i, \ast_i)$ be a finite connected folded $A$-graph with $b(\Gamma_i)$ representing $H_i$ such that $(\Gamma_i, \ast_i)$ has the smallest number of (topological) edges among all such $A$-graphs. Thus $\Gamma_i$ is folded and  has no degree-1 vertices except possibly $\ast_i$. Moreover, by Proposition~\ref{p:red} and Corollary~\ref{c:red}, the label of every reduced edge-path in $\Gamma_i$ is $\lambda$-reduced and $(C,0)$-quasigeodesic in $G$. 
Recall that according to our convention for $A$-graphs, oriented edges in $\Gamma_i$ labelled by elements of $A$ are considered positive, and oriented edges in $\Gamma_i$ labelled by elements of $A^{-1}$ are considered negative. Recall also that $vol(\Gamma_i)=\# E_+\Gamma_i$, the number of topological edges of $\Gamma_i$. Since $(\ddag)$ is an infinite sequence of distinct $k$-generated subgroups of $G$, there are infinitely many distinct basepointed $A$-graphs in the sequence $(\Gamma_i, \ast_i)_{i=1}^\infty$. After passing to a subsequence, we will assume that $vol(\Gamma_i)<vol(\Gamma_{i+1})$ for all $i\ge 1$. 

For the graph $(\Gamma_1, \ast_1)$ choose a maximal tree $T\subseteq \Gamma$. Each positive edge in $\Gamma_i-T$ defines an element of $\pi_1(\Gamma_i,\ast_i)$ and the set of all such elements gives a free basis $B_T$ of $\pi_1(\Gamma_i,\ast_i)$. Denote by $S_T$ the set of labels of the paths in $B_T$. Thus $S_T$ is a free basis of $H_1$. By construction, $\#B_T=\#S_T=b(\Gamma_1)\le k$ and for every $w\in S_T$ we have $|w|\le 2 vol(\Gamma_1)$. 

For every $i=2, 3, 4, \dots$ and for every $w\in S_T$ choose a freely reduced word $\hat w(i)\in F(A)$ such that $\hat w(i)=G w$ and that $\hat w_i$ labels a closed reduced path $p(w,i)$ from $\ast_i$ to $\ast_i$ in $\Gamma_i$. Such $\hat w_i$ exists since $H_1\le H_i$. Since $\hat w(i)$ is a $(C,0)$-quasigeodesic in $G$, we have $|\hat w_i|=|p(w,i)|\le C|w|\le 2C\, vol(\Gamma_1)$.  For all $i\ge 2$, let $\Delta_{i,1}$ be the subgraph of $\Gamma_i$ spanned by the union of $p(w,i)$ over all $w\in S_T$. Thus $\Delta_{i,1}$ is a finite connected graph containing $\ast_i$ with $vol(\Delta_{i,1})\le 2kC\, vol(\Gamma_1)$. Thus there are only finitely many possibilities for the base-pointed $A$-graph $(\Delta_{i,1},\ast_i)$. After passing to a further subsequence, we will assume that there is a finite connected base-pointed $A$-graph $(\Delta_1,x_1)$ such that for all $i\ge 2$ we have $(\Delta_{i,1},\ast_i)=(\Delta_1,x_1)$. Note that, since $H_1\ne 1$, the graph $\Delta_1$ is non-contractible, so that $b(\Delta_1)\ge 1$. 

Since all the subgroups $H_1, H_2,\dots, $ are distinct, for infinitely many $i\ge 2$ we have $\Delta_{i,1}\subsetneq \Gamma_i$. After passing to a further subsequence, we may assume that $\Delta_{i,1}\subsetneq \Gamma_i$ for all $i\ge 2$. 

Since $\Delta_{2,1}\subsetneq \Gamma_2$, we can choose a finite connected subgraph $\Gamma_2'\subseteq \Gamma_2$ such that $\Delta_{2,1}\subseteq \Gamma_2'$, that $b(\Gamma_2')=b(\Delta_{2,1})+1=b(\Delta_2)+1$ and that $\Gamma_2'$ has no degree-1 vertices except possibly for $\ast_2$. Thus $\Gamma_2'$ is spanned by the union of $\Delta_{2,1}$ and a single closed reduced path $\gamma_2$ at $\ast_2$, not contained in $\Delta_{2,1}$. We have $\pi_1(\Gamma_2',\ast_2)=\pi_1(\Delta_{2,1},\ast_2)\ast \langle \gamma_2\rangle$.  Let $u$ be the label of $\gamma_2$.

Now for each $i=3, 4, 5, \dots$ choose a reduced word $u_i$ labeling a closed path $p(\gamma_2,i)$ in $\Gamma_i$ from $\ast_i$ to $\ast_i$ such that $u_i=_G u$.  Then $|p(\gamma_2,i)|=|u_i|\le C |u|=C|\gamma_2|$. For $i=3, 4, 5, \dots $ let $\Delta{i,2}$ be the subgraph of $\Gamma_i$ given by the union of $\Delta_{i,1}$ and the path $p(\gamma_2,i)$. Since all the labeling homomorphisms for the graphs under consideration are injective and $\gamma_2\not\in \pi_1(\Delta_{2,1})=\pi_1(\Delta_1)$, and since $\Delta_{i,1}=\Delta_1$ for all $i\ge 3$, we have $p(\gamma_2,i)\not\subseteq \Delta_{i,1}$ and thus $\Delta_{i,1}\subsetneq \Delta_{i,2}$ for all $i\ge 3$. Therefore $b(\Delta_{i,2})\ge b(\Delta_1)+1$ for all $i\ge 3$. Moreover, by construction $vol(\Delta_{i,2})\le vol(\Delta_1)+C|\gamma_2|$.  Thus there are only finitely many choices for the base-pointed $A$-graph $(\Delta_{i,2},\ast_i)$, where $i\ge 3$. After passing to a further subsequence, we may assume that there is a finite connected base-pointed $A$-graph $(\Delta_2,\ast)$ such that for all $i\ge 3$ we have $(\Delta_{i,2},\ast_2)=(\Delta_2,\ast)$. 

Iterating this process, after passing to a subsequence of the original sequence, for $j=1, 2, 3,, \dots$ we construct an infinite sequence of finite connected base-pointed $A$-graphs $(\Delta_j, x_j)$ and subgraphs $\Delta_{i,j} \subseteq \Gamma_{i}$, where $i\ge j+1$, with the following properties:

\begin{enumerate}
\item Each $\Delta_{i,j}$ contains $\ast_i$ and satisfies $(\Delta_j, x_j)=(\Delta_{i,j} ,\ast_i)$ where $j\ge 1$ and $i\ge j+1$.
\item We have $b(\Delta_{j+1})\ge b(\Delta_j)+1$ for $j=1,2,3,\dots$.
\item We have $(\Delta_j, x_j)\subsetneq (\Delta_{j+1}, x_{j+1})$ for all $j\ge 1$.
\item We have $b(\Delta_1)\ge 1$. 
\end{enumerate}

Then $b(\Delta_{k+1})\ge k+1$. Since $(\Delta_{k+2,k+1},\ast_{k+2})=(\Delta_{k+1}, x_{k+1})$ is a connected subgraph of $\Gamma_{k+2}$, it follows that $b(\Gamma_{k+2})\ge k+2$, yielding a contradiction.
\end{proof}

\begin{proof}[Proof of Theorem~\ref{t:A}]

We may assume that $k\ge m$ since $ACC_k$ implies $ACC_s$ for every integer $1\le s \le k$.

Choose $0<\lambda<\mu<1/6$ so that $\lambda,\mu$ satisfy the assumptions of Theorem~\ref{t:ACC}. Put $Q_{m,t,k}=Q_{m.t}(\lambda,\mu,k)$. We claim that  $Q_{m,t,k}$ satisfies the conclusions of Theorem~\ref{t:A}.

First, by Proposition~\ref{p:generic} the set $Q_{m.t}(\lambda,\mu,k)$ is exponentially generic in the set of all $m$-generator $t$-relator presentations $\langle a_1,\dots, a_k| r_1,\dots, r_t\rangle$ with $r_i$ being cyclically reduced in $F(A)$.  Theorem~\ref{t:ACC} shows that $ACC_k$ holds for every group $G$, given by a $Q_{m.t}(\lambda,\mu,k)$  presentation. Therefore $Q_{m,t,k}=Q_{m.t}(\lambda,\mu,k)$ satisfies the conclusions of Theorem~\ref{t:A}, as required.

\end{proof}

\begin{proof}[Proof of Theorem~\ref{t:B}]

As before, it is enough to prove Theorem~\ref{t:B} for all $k\ge m$, and thus we assume that $k\ge m$.

Choose $0<\lambda,\mu<1/6$ as in Theorem~\ref{t:ACC}.

It is known~\cite{Oll07} that for any $0<\lambda\le 1/6$ the the $C'(\lambda)$ small cancellation condition is satisfied with overwhelming probability for $d$-random groups for sufficiently small density $d>0$. Also, by Proposition~\ref{p:red}, for any $k\ge m\ge 2$ condition $Q_{m,1}(\lambda,\mu,k)$ for cyclically reduced words in $F_m=F(a_1,\dots, a_m)$ is exponentially generic, for $\mu,\lambda$ as in Proposition~\ref{p:red}. Therefore, by Corollary~4.3 of \cite{KS08}, for any $k\ge m\ge 2$ and $\lambda\mu$ as in Proposition~\ref{p:red}, there exists $0<d_0=d_0(m,k)<1$ such that for any $0<d<d_0$ condition $Q_{m}(\lambda,\mu,k)$ holds with overwhelming probability for random finite presentations $\langle a_1,\dots, a_m|R\rangle$ at density $d$.  By Theorem~\ref{t:ACC}, $ACC_k$ holds for every group given by a $Q_{m}(\lambda,\mu,k)$ presentation. The conclusion of Theorem~\ref{t:B} now follows.
\end{proof}

\end{document}